\documentclass{amsart}
\usepackage[all]{xy}
\usepackage{verbatim}
\usepackage{color}
\usepackage{amsthm}
\usepackage{amssymb}
\usepackage[colorlinks=true]{hyperref}
\usepackage{footmisc}

\numberwithin{equation}{section}

\newtheorem{theorem}[equation]{Theorem}
\newtheorem*{theorem*}{Theorem} \newtheorem{lemma}[equation]{Lemma}

\newtheorem*{conjecture*}{Mamma Conjecture}
\newtheorem*{conjecture1*}{Mamma Conjecture (revisited)}
\newtheorem{proposition}[equation]{Proposition}
\newtheorem{corollary}[equation]{Corollary}
\newtheorem*{corollary*}{Corollary}

\theoremstyle{remark}

\newtheorem{example}[equation]{Example}

\newtheorem{notation}[equation]{Notation}

\theoremstyle{remark}
\newtheorem{remark}[equation]{Remark}

\newcommand{\cA}{{\mathcal A}}
\newcommand{\cB}{{\mathcal B}}
\newcommand{\cC}{{\mathcal C}}
\newcommand{\cD}{{\mathcal D}}
\newcommand{\cE}{{\mathcal E}}

\newcommand{\cG}{{\mathcal G}}

\newcommand{\cL}{{\mathcal L}}

\newcommand{\cO}{{\mathcal O}}

\newcommand{\cT}{{\mathcal T}}

\newcommand{\Spt}{\mathrm{Spt}}

\newcommand{\bbF}{\mathbb{F}}

\newcommand{\bbN}{\mathbb{N}}
\newcommand{\bbP}{\mathbb{P}}

\newcommand{\bbS}{\mathbb{S}}

\newcommand{\bbZ}{\mathbb{Z}}

\DeclareMathOperator{\id}{id}

\DeclareMathOperator{\NMot}{NMot}

\newcommand{\dgcat}{\mathrm{dgcat}} 

\newcommand{\bbK}{I\mspace{-6.mu}K}

\newcommand{\dg}{\mathrm{dg}}

\newcommand{\Hom}{\mathrm{Hom}}

\newcommand{\Ho}{\mathrm{Ho}}

\newcommand{\too}{\longrightarrow}

\newcommand{\ie}{\textsl{i.e.}\ }
\newcommand{\eg}{\textsl{e.g.}}

\let\oldmarginpar\marginpar
\def\marginpar#1{\oldmarginpar{\tiny #1}}

\begin{document}

\title[Invariants of noncommutative projective schemes]{Invariants of noncommutative projective schemes}
\author{Gon{\c c}alo~Tabuada}

\address{Gon{\c c}alo Tabuada, Department of Mathematics, MIT, Cambridge, MA 02139, USA}
\email{tabuada@math.mit.edu}
\urladdr{http://math.mit.edu/~tabuada}
\thanks{The author was partially supported by a NSF CAREER Award}

\date{\today}

\abstract{In this note we compute several invariants (\eg\ algebraic $K$-theory, cyclic homology and topological Hochschild homology) of the noncommutative projective schemes associated to Koszul algebras of finite global~dimension.}}

\maketitle
\vskip-\baselineskip
\vskip-\baselineskip


\section{Introduction}\label{sec:intro}
\subsection*{Noncommutative projective schemes}
Let $k$ be a field and $A=\bigoplus_{n\geq 0} A_n$ a $\bbN$-graded Noetherian $k$-algebra. Throughout the note, we will always assume that $A$ is {\em connected}, \ie $A_0=k$, and {\em locally finite-dimensional}, \ie $\mathrm{dim}_k(A_n)<\infty$ for every $n$. Following Manin \cite{Manin}, Gabriel \cite{Gabriel}, Artin-Zhang \cite{Artin-Zhang}, and others, the {\em noncommutative projective scheme $\mathrm{qgr}(A)$} associated to $A$ is defined as the quotient category $\mathrm{gr}(A)/\mathrm{tors}(A)$, where $\mathrm{gr}(A)$ stands for the abelian category of finitely generated $\bbZ$-graded (right) $A$-modules and $\mathrm{tors}(A)$ for the Serre subcategory of torsion $A$-modules. This definition was motivated by Serre's celebrated result \cite[Prop.~7.8]{Serre}, which asserts that in the particular case where $A$ is commutative and generated by elements of degree $1$ the quotient category $\mathrm{qgr}(A)$ is equivalent to the abelian category of coherent $\cO_{\mathrm{Proj}(A)}$-modules $\mathrm{coh}(\mathrm{Proj}(A))$. For example, when $A$ is the polynomial $k$-algebra $k[x_1, \ldots, x_d]$, with $\mathrm{deg}(x_i)=1$, we have the following equivalence $\mathrm{qgr}(k[x_1, \ldots, x_d])\simeq \mathrm{coh}(\bbP^{d-1})$.
\subsection*{Invariants of dg categories}
A {\em dg category} $\cA$ is a category enriched over complexes of $k$-vector spaces; consult Keller's survey \cite{ICM-Keller}. Every (dg) $k$-algebra $B$ gives naturally rise to a dg category with a single object. Another source of examples is provided by exact categories since the bounded derived category $\cD^b(\cE)$ of every exact category $\cE$ admits a canonical dg enhancement $\cD^b_\dg(\cE)$; see \cite[\S4.4]{ICM-Keller}. In what follows, we will denote by $\dgcat(k)$ the category of dg categories and dg functors. A functor $E\colon\dgcat(k) \to \cT$, with values in a triangulated category,~is~called:
\begin{itemize}
\item[(i)] {\em Morita invariant} if it inverts the Morita equivalences; see \cite[\S4.6]{ICM-Keller}.
\item[(ii)] {\em Localizing} if it sends short exact sequences of dg categories, in the sense of Drinfeld/Keller (see \cite{Drinfeld}\cite[\S4.6]{ICM-Keller}), to distinguished triangles:
\begin{eqnarray*}
0 \too \cA \too \cB \too \cC \too 0 &\mapsto & E(\cA) \too E(\cB) \too E(\cC) \stackrel{\partial}{\too} \Sigma E(\cA)\,.
\end{eqnarray*}
\item[(iii)] {\em Co-continuous} if it preserves sequential (homotopy) colimits.\end{itemize}
Examples of functors satisfying the conditions (i)-(iii) include nonconnective algebraic $K$-theory $\bbK$, homotopy $K$-theory $KH$, \'etale $K$-theory $K^{\mathrm{et}}$, the mixed complex $C$, Hochschild homology $HH$, cyclic homology $HC$, and topological Hochschild homology $THH$; see \cite[\S8.2]{book}. Some other functors such as periodic cyclic homology $HP$ and negative cyclic homology $HN$ only satisfy conditions (i)-(ii). When applied to $B$, resp. to  $\cD^b_\dg(\cE)$, all the preceding invariants of dg categories reduce to the corresponding invariants of the (dg) $k$-algebra $B$, resp. of the exact category~$\cE$.
\begin{notation}
Given a functor $E\colon \dgcat(k) \to \cT$, an object $o \in \cT$, an integer $m \in \bbZ$, and a dg category $\cA$, let us write $E_m^o(\cA) := \Hom_\cT(\Sigma^m(o), E(\cA))$. Whenever $\cT$ is symmetric monoidal with $\otimes$-unit ${\bf 1}$, we will write $E_m(\cA)$ instead of $E_m^{{\bf 1}}(\cA)$.
\end{notation}
\subsection*{Statement of results}
Let $k$ be a field and $A=\bigoplus_{n\geq 0} A_n$ a $\bbN$-graded Noetherian $k$-algebra. Assume that $A$ is Koszul and has finite global dimension $d$. Under these assumptions, the Hilbert series $h_A(t):=\sum_{n\geq 0} \mathrm{dim}_k(A_n)t^n \in \bbZ[\![t]\!]$ is invertible and its inverse $h_A(t)^{-1}$ is a polynomial $1-\beta_1t + \beta_2t^2 - \cdots + (-1)^d\beta_d t^d$ of degree $d$, where $\beta_i$ stands for the dimension of the $k$-vector space $\mathrm{Tor}^A_i(k,k)$ (or $\mathrm{Ext}_A^i(k,k)$). In what follows, we write $\beta:=\beta_d$. Our main result is the following computation: 
\begin{theorem}\label{thm:main}
Let $A$ be a $k$-algebra as above and $E\colon \dgcat(k) \to \cT$ a functor satisfying conditions (i)-(iii). Assume that $\cT$ is $R$-linear for a commutative~ring~$R$.
\begin{itemize}
\item[(i)] For every compact object $o \in \cT$, we have $R$-module isomorphisms
\begin{eqnarray}
E^o_m(\cD^b_\dg(\mathrm{qgr}(A)))\simeq R[t]/\langle h'_A(t)^{-1}\rangle \otimes_R E^o_m(k) && m \in \bbZ\,,
\end{eqnarray}
where $h'_A(t)^{-1}=1-\beta'_1t + \beta'_2t^2 - \cdots + (-1)^{d'}\beta'_{d'} t^{d'}$ stands for the image of the polynomial $h_A(t)^{-1}$ in $R[t]$. In what follows, we write $\beta':=\beta'_{d'}$.
\item[(ii)] Assume moreover that $1/\beta' \in R$ and that $\cT$ is compactly generated. Under these assumptions, we have an isomorphism $E(\cD^b_\dg(\mathrm{qgr}(A)))\simeq E(k)^{\oplus d'}$.
\end{itemize}
\end{theorem}
\begin{remark}\label{rk:new}
\begin{itemize}
\item[(i)] If $\beta=1$, then $\beta'=\beta$ and $d'=d$. As proved in \cite[Cor.~0.2]{Zhang}, in the particular case where $d=3$, we always have $h_A(t)^{-1}=(1-t)^3$.
\item[(ii)] If $R$ is a field, then $1/\beta' \in R$. Moreover, $\beta'=\beta$ and $d'=d$ if and only if the characteristic of $R$ does not divides $\beta$.
\end{itemize}
\end{remark}
\begin{corollary}\label{cor:main}
Let $A$ be a $k$-algebra as above and $E\colon \dgcat(k) \to \cT$ a functor satisfying conditions (i)-(iii). Assume moreover that $\cT$ is compactly generated. Under these assumptions, we have an isomorphism $E(\cD^b_\dg(\mathrm{qgr}(A)))_{1/\beta'}\simeq E(k)_{1/\beta'}^{\oplus d'}$ in the $\bbZ[1/\beta']$-linearized triangulated category\footnote{Let $\cG$ be a set of compact generators of $\cT$. Recall that $\cT_{1/\beta'}$ may be defined as the Verdier quotient of $\cT$ by the smallest localizing (=closed under arbitrary direct sums) triangulated subcategory containing the objects $\{\mathrm{cone}(\beta\cdot  \id_o)\,|\, o \in \cG\}$.} $\cT_{1/\beta'}$.
\end{corollary}
\begin{proof}
By construction, the triangulated category $\cT_{1/\beta'}$ is compactly generated and $R[1/\beta']$-linear. Moreover, the $\bbZ[1/\beta']$-linearization functor $(-)_{1/\beta'}\colon \cT \to \cT_{1/\beta'}$ is triangulated and preserves arbitrary direct sums. Therefore, the proof follows from Theorem \ref{thm:main}(ii) applied to $E=E(-)_{1/\beta'}$ (with $R=R[1/\beta']$).
\end{proof}
\begin{example}[Algebraic $K$-theory]\label{ex:1}
Nonconnective algebraic $K$-theory gives rise to a functor $\bbK\colon \dgcat(k) \to \Ho(\Spt)$, with values in the homotopy category of spectra, satisfying conditions (i)-(iii); see \cite[\S8.2.1]{book}. Therefore, by applying Theorem \ref{thm:main}(i) to $E=\bbK$ (with $R=\bbZ$) and to the sphere spectrum $o=\bbS$, we obtain isomorphisms\footnote{In the particular case where $m=0$, the isomorphism \eqref{eq:iso-K} was originally established by Mori-Smith in \cite[Thm.~2.3]{MS}.}
\begin{eqnarray}\label{eq:iso-K}
\bbK_m(\mathrm{qgr}(A)) \simeq \bbZ[t]/\langle h_A(t)^{-1}\rangle \otimes_\bbZ \bbK_m(k) && m \in \bbZ\,.
\end{eqnarray}
Moreover, since the triangulated category $\Ho(\Spt)$ is compactly generated, Corollary \ref{cor:main} implies that $\bbK(\mathrm{qgr}(A))_{1/\beta}\simeq \bbK(k)^{\oplus d}_{1/\beta}$. All the above holds {\em mutatis mutandis} with $\bbK$ replaced by $KH$ or $K^{\mathrm{et}}$.
\end{example}
\begin{example}[Mixed complex]\label{ex:2}
Following Kassel \cite{Kassel}, a {\em mixed complex} is a (right) dg module over the $k$-algebra of dual numbers $\Lambda :=k[\epsilon]/\epsilon^2$ with $\mathrm{deg}(\epsilon)=-1$ and $d(\epsilon)=0$. The mixed complex gives rise to a functor $C\colon \dgcat(k) \to \cD(\Lambda)$, with values in the derived category of $\Lambda$, satisfying conditions (i)-(iii); see \cite[\S8.2.4]{book}. Therefore, since the category $\cD(\Lambda)$ is compactly generated, by applying Theorem \ref{thm:main}(ii) to $E=C$ (with $R=k$), we obtain an isomorphism $C(\mathrm{qgr}(A))\simeq C(k)^{\oplus d'}$.
\end{example}
\begin{example}[Cyclic homology and its variants]\label{ex:3}
As explained by Keller in \cite[\S2.2]{Keller1}, Hochschild homology $HH$, cyclic homology $HC$, periodic cyclic homology $HP$, and negative cyclic homology $HN$, can be recovered from the mixed complex $C$. Therefore, making use of Example \ref{ex:2}, we conclude that
\begin{eqnarray*}
HH(\mathrm{qgr}(A))\simeq HH(k)^{\oplus d'} && HC(\mathrm{qgr}(A))\simeq HC(k)^{\oplus d'}  \\
HP(\mathrm{qgr}(A))\simeq HP(k)^{\oplus d'} && HN(\mathrm{qgr}(A))\simeq HN(k)^{\oplus d'} \,.
\end{eqnarray*}
\end{example}
\begin{example}[Topological Hochschild homology]\label{ex:4}
Topological Hochschild homology gives rise to a (lax symmetric monoidal) functor $THH\colon \dgcat(k) \to \Ho(\Spt)$ satisfying conditions (i)-(iii); see \cite[\S8.2.8]{book}. Since the ``inclusion of the $0^{\mathrm{th}}$ skeleton'' yields a ring homomorphism $k \to THH_0(k)$, the abelian groups $THH_\ast$ are then naturally equipped with a $k$-linear structure. Therefore, using the fact that the triangulated category $\Ho(\Spt)$ is (compactly) generated by the sphere spectrum $\bbS$, an argument similar to the one used in the proof of Theorem \ref{thm:main}(ii) allows us to conclude that $THH(\mathrm{qgr}(A))\simeq THH(k)^{\oplus d'}$. For example, in the particular where $k=\bbF_p$, with $p$ a prime number, we have the following isomorphisms:
$$ THH_m(\mathrm{qgr}(A))\simeq \left\{ \begin{array}{ll}
(\bbF_p)^{\oplus d'} &   m\geq 0 \,\,\, \mathrm{even} \\
0 &  \text{otherwise}\,.
\end{array}\right.
$$
\end{example}
Intuitively speaking, Theorem \ref{thm:main} (as well as Corollary \ref{cor:main} and Examples \ref{ex:1}-\ref{ex:4}) shows that all the different invariants of a noncommutative projective scheme $\mathrm{qgr}(A)$ are completely determined by the Hilbert series $h_A(t)$~of~$A$. 


Theorem \ref{thm:main} (as well as Corollary \ref{cor:main}) may be applied to the following algebras:
\begin{example}[Quantum polynomial algebras]
Choose constant elements $q_{ij} \in k^\times$ with $1 \leq i < j \leq d$. The following $\bbN$-graded Noetherian $k$-algebra
$$ A:=k\langle x_1, \ldots, x_d\rangle/\langle x_j x_i - q_{ij} x_i x_j\,|\, 1 \leq i < j \leq d\rangle\,,$$
with $\mathrm{deg}(x_i)=1$, is called the {\em quantum polynomial algebra} associated to $q_{ij}$. This algebra is Koszul,  has global dimension $d$, and $h_A(t)^{-1}=(1-t)^d$; see \cite[\S1]{Manin-Fourier}. 
\end{example}
\begin{example}[Quantum matrix algebras]
Choose a $q \in k^\times$. The $\bbN$-graded Noetherian $k$-algebra $A$ defined as the quotient of $k\langle x_1, x_2, x_3, x_4\rangle$~by~the~relations
\begin{eqnarray*}
x_1x_2 = q x_2 x_1 & x_1x_3=qx_3 x_1 &  x_1 x_4 - x_4 x_1 = (q-q^{-1}) x_2 x_3 \\
x_2 x_3 = x_3 x_2 &  x_2 x_4=q x_4 x_2 & x_1 x_4 = q x_4 x_3\,,
\end{eqnarray*}
with $\mathrm{deg}(x_i)=1$, is called the {\em quantum matrix algebra} associated to $q$. This algebra is Koszul, has global dimension $4$, and $h_A(t)^{-1}=(1-t)^4$; see \cite[\S1]{Manin-Fourier}. 
\end{example}
\begin{example}[Sklyanin algebras]
Let $C$ be a smooth elliptic curve, $\sigma \in \mathrm{Aut}(C)$ an automorphism given by translation under the group law, and $\cL$ a line bundle on $C$ of degree $d\geq 3$. We write $\Gamma_\sigma \subset C\times C$ for the graph of $\sigma$ and $V$ for the $d$-dimensional $k$-vector space $H^0(C,\cL)$. The $\bbN$-graded Noetherian $k$-algebra $A:=T(V)/R$, where
$$R:=H^0(C\times C, (\cL\boxtimes \cL)(-\Gamma_\sigma))\subset H^0(C\times C, \cL \boxtimes \cL)=V\otimes V\,,$$
is called the {\em Sklyanin algebra} associated to the triple $(C,\sigma, \cL)$. This algebra is Koszul, has global dimension $d$, and $h_A(t)^{-1}=(1-t)^d$; see \cite{Feigin}\cite[\S1]{TateVdb}.
%
%
\end{example}
\begin{example}[Homogenized enveloping algebras]
Let $\mathfrak{g}$ be a finite dimensional Lie algebra. The following $\bbN$-graded Noetherian $k$-algebra ($z$ is a new variable)
$$ A:= T(\mathfrak{g} \oplus kz)/ \langle\{z\otimes x - x\otimes z \,|\, x \in \mathfrak{g}\}\cup \{x\otimes y - y\otimes x - [x,y]\otimes z \,|\, x, y \in \mathfrak{g}\}\rangle\,,$$
is called the {\em homogenized enveloping algebra} of $\mathfrak{g}$. This algebra is Koszul, has global dimension $d:=\mathrm{dim}(\mathfrak{g})+1$, and $h_A(t)^{-1}=(1-t)^{d}$; see \cite[\S12]{Smith}.
\end{example}
\begin{example}\label{ex:key}
Let $k$ be an uncountable algebraically closed field. Choose a pair of elements $(\theta, \rho)$ of $k^\times$ which are algebraically independent over the prime field of $k$ and write $\Theta:=\frac{\theta-1}{\theta+1}$ and $\Delta:=\frac{\rho-1}{\rho+1}$. Under these assumptions and notations, the $\bbN$-graded Noetherian $k$-algebra $A:=k\langle x_1, x_2, x_3, x_4\rangle/\langle f_1, \ldots, f_6\rangle$, where 
\begin{eqnarray*}
f_1 :=  x_1(\Theta x_1 - x_3) + x_3(x_1 - \Theta x_3) &&
f_2  :=  x_1(\Theta x_2 - x_4) + x_3(x_2 - \Theta x_4) \\
f_3 :=  x_2(\Theta x_1 - x_3) + x_4(x_1 - \Theta x_3) &&
f_4  :=  x_2(\Theta x_2 - x_4) + x_4(x_2 - \Theta x_4) \\
f_5 :=  x_1(\Delta x_1 - x_2) + x_4(x_1 - \Delta x_2) &&
f_6  :=  x_1(\Delta x_3 - x_4) + x_4(x_3 - \Delta x_4)\,,
\end{eqnarray*}
is Koszul, has global dimension $4$, and $h_A(t)^{-1}=(1-t)^{4}$; see \cite[Thm.~3.5]{Sierra}. 
\end{example}
\subsection*{Gorenstein algebras}
Recall that a $\bbN$-graded Noetherian $k$-algebra $A=\bigoplus_{n\geq 0} A_n$ is called {\em Gorenstein}, with Gorenstein parameter $l$, if it has finite injective dimension $m$ and ${\bf R}\mathrm{Hom}_A(k,A)\simeq \Sigma^{-m}k(l)$, where $k(l)$ stands for the $\bbZ$-graded (right) $A$-module $k(l)_n:=k_{n+l}$. Let us assume moreover that $A$ has finite global dimension $d$; this implies that $d=m$. Under these assumptions, a remarkable result of Orlov (see \cite[Cor.~2.7]{Orlov}) asserts that the bounded derived category $\cD^b(\mathrm{qgr}(A))$ admits a full exceptional collection of length $l$. This leads naturally to the following result:
\begin{theorem}\label{thm:Orlov}
Let $A$ be a $\bbN$-graded Noetherian $k$-algebra and $E$ a functor satisfying conditions\footnote{More generally, condition (ii) can be replaced by {\em additivity} in the sense of \cite[Def.~2.1]{book}.} (i)-(ii). Assume that $A$ is Gorenstein, with Gorenstein parameter $l$, and has finite global dimension $d$. Under these assumptions, we have an isomorphism $E(\cD^b_\dg(\mathrm{qgr}(A)))\simeq E(k)^{\oplus l}$.
\end{theorem}
\begin{proof}
As explained in \cite[\S2.4.2 and \S8.4.5]{book}, every functor $E$ satisfying conditions (i)-(ii) sends a full exceptional collections of length $l$ to the direct sum $E(k)^{\oplus l}$.
\end{proof}
\begin{remark}
\begin{itemize}
\item[(i)] Since $A$ is connected and has finite global dimension, the Hilbert series $h_A(t)$ is invertible and its inverse $h_A(t)^{-1}$ is a polynomial. Moreover, the Gorenstein condition implies that $h_A(t)^{-1}$ is monic and has degree $l$.
\item[(ii)] As proved in \cite[Chap.~2\, Thm.~2.5]{Quadratic}, $A$ is moreover Koszul if and only if $d=l$.
\end{itemize}
\end{remark}
Note that Theorem \ref{thm:main} does {\em not} follows from Theorem \ref{thm:Orlov} because, in general, Koszulness does {\em not} implies\footnote{In the particular case where $d=3$, Koszulness indeed implies Gorensteiness; see \cite[Cor.~0.2]{Zhang}.} Gorensteiness. For instance, the algebras $A$ of Example \ref{ex:key} are Koszul but {\em not} Gorenstein; see \cite[Thm.~3.5]{Sierra}. In this latter example, we have moreover $\mathrm{dim}_k(\mathrm{Ext}^i_A(k,A))=\infty$ for $i=2,3,4$; see \cite[Prop.~5.11]{Sierra}. Consequently, the $k$-linear triangulated categories $\cD^b(\mathrm{qgr}(A))$ are {\em not} even Ext-finite.
\section{Proof of Theorem \ref{thm:main}}
Recall from Quillen \cite[\S2]{Quillen} that an exact category $\cE$ is an additive category equipped with a family of short exact sequences satisfying some standard conditions. In order to simplify the exposition, given an exact functor $F\colon \cE \to \cE'$, we will still denote by $F\colon \cD^b_\dg(\cE) \to \cD^b_\dg(\cE')$ the induced dg functor. We start with the following general result of independent interest:
\begin{proposition}\label{prop:additive}
Let $0 \to F_1 \to F_2 \to F_3 \to 0$ be a short exact sequence of exact functors $F_1, F_2, F_3\colon \cE \to \cE'$. Given any localizing functor $E\colon \dgcat(k) \to \cT$, we have the following equality $E(F_2) = E(F_1) + E(F_3)$.
\end{proposition}
\begin{proof}
Let $\mathrm{Ex}(\cE')$ be the category of short exact sequences $\varepsilon=(a \to b \to c)$ in $\cE'$; this is also an exact category with short exact sequence defined componentwise. By construction, $\mathrm{Ex}(\cE')$ comes equipped with the following exact functors 
\begin{eqnarray*}
\iota_1\colon \cE' \too \mathrm{Ex}(\cE') \quad a \mapsto (a = a \to 0) && \iota_2\colon \cE' \too \mathrm{Ex}(\cE')\quad a\mapsto (0 \to a = a)
\end{eqnarray*}
\begin{eqnarray*}
\pi_1\colon \mathrm{Ex}(\cE') \stackrel{\varepsilon \mapsto a}{\too} \cE' & \pi_2\colon \mathrm{Ex}(\cE') \stackrel{\varepsilon \mapsto b}{\too} \cE' & \pi_3\colon \mathrm{Ex}(\cE') \stackrel{\varepsilon \mapsto c}{\too} \cE'
\end{eqnarray*}
satisfying the equalities $\pi_1 \circ \iota_1 = \pi_2 \circ \iota_1 = \id$, $\pi_3 \circ \iota_1 = \pi_1 \circ \iota_2 =0$, and $\pi_2 \circ \iota_2 = \pi_3 \circ \iota_2 = \id$. Moreover, we have the following short exact sequence of dg categories
$$ 0 \too \cD^b_\dg(\cE') \stackrel{\iota_1}{\too} \cD^b_\dg(\mathrm{Ex}(\cE')) \stackrel{\pi_3}{\too} \cD^b_\dg(\cE') \too 0$$
and consequently the following distinguished triangle
\begin{equation*}\label{eq:triangle-last}
E(\cD^b_\dg(\cE'))\stackrel{E(\iota_1)}{\too} E(\cD^b_\dg(\mathrm{Ex}(\cE'))) \stackrel{E(\pi_3)}{\too} E(\cD^b_\dg(\cE')) \stackrel{\partial}{\too} \Sigma E(\cD^b_\dg(\cE'))\,.
\end{equation*}
Since $\pi_3 \circ \iota_2 =\id$, the preceding triangle splits an induces an isomorphism
\begin{equation}\label{eq:isom}
[E(\iota_1)\,\, E(\iota_2)]\colon E(\cD^b_\dg(\cE')) \oplus E(\cD^b_\dg(\cE')) \stackrel{\simeq}{\too} E(\cD^b_\dg(\mathrm{Ex}(\cE')))\,.
\end{equation}
Note that a short exact sequence of exact functors $0 \to F_1 \to F_2 \to F_3 \to 0$~is~the same data as an exact functor $F\colon \cE \to \mathrm{Ex}(\cE')$. Therefore, by combining the equalities $E(\pi_2) \circ [E(\iota_1)\,\, E(\iota_2)] = [\id\,\, \id]$ and $(E(\pi_1)+E(\pi_3))\circ [E(\iota_1)\,\, E(\iota_2)] = [\id\,\, \id]$ with the fact that \eqref{eq:isom} is an isomorphism, we conclude that $E(\pi_2) = E(\pi_1) +E(\pi_3)$. The proof follows now from the equalities $\pi_1 \circ F=F_1$, $\pi_2 \circ F= F_2$, and $\pi_3 \circ F= F_3$. 
\end{proof}
Let $B=\bigoplus_{n \geq 0}B_n$ be a $\bbN$-graded $k$-algebra and $\mathrm{grproj}(B)$ the exact category of finitely generated projective $\bbZ$-graded (right) $B$-modules. The following general computation is also of independent interest:
\begin{proposition}\label{prop:pgr}
We have an isomorphism $E(\cD_\dg^b(\mathrm{grproj}(B)))\simeq \oplus_{-\infty}^{+\infty} E(B_0)$.
\end{proposition}
\begin{proof}
Consider $B_0$ as an $\bbN$-graded $k$-algebra concentrated in degree zero. The canonical inclusion $B_0 \to B$ and projection $B\to B_0$ of $\bbN$-graded $k$-algebras give rise to the following base-change exact functors:
\begin{eqnarray*}
\varphi\colon \mathrm{grproj}(B_0) \too \mathrm{grproj}(B) && P\mapsto P\otimes_{B_0}B \\
\psi\colon \mathrm{grproj}(B) \too \mathrm{grproj}(B_0) && P\mapsto P\otimes_{B}B_0\,.
\end{eqnarray*}
Since $\psi\circ \varphi= \id$, it follows from Lemma \ref{lem:aux} below that $\varphi$ and $\psi$ give rise to inverse isomorphisms between $E(\cD^b_\dg(\mathrm{grproj}(B)))$ and $E(\cD^b_\dg(\mathrm{grproj}(B_0)))$. 

Now, note that we have the following canonical equivalence of exact categories
\begin{eqnarray}\label{eq:equivalence-exact}
\mathrm{grproj}(B_0) \stackrel{\simeq}{\too} \amalg_{n \in \bbZ} \mathrm{proj}(B_0) && P \mapsto \{P_n\}_{n \in \bbZ}\,,
\end{eqnarray}
where $\mathrm{proj}(B_0)$ stands for the exact category of finitely generated projective (right) $B_0$-modules. Since the dg category $\cD^b_\dg(\mathrm{proj}(B_0))$ is Morita equivalent to the $k$-algebra $B_0$ and the functor $E$ is co-continuous, we then conclude from the equivalence \eqref{eq:equivalence-exact} that $E(\cD^b_\dg(\mathrm{grproj}(B_0)))\simeq \oplus^{+ \infty}_{-\infty} E(B_0)$. This finishes the proof.
\end{proof}
\begin{lemma}\label{lem:aux}
The following endomorphism is equal to the identity
$$ E(\varphi\circ \psi)\colon E(\cD^b_\dg(\mathrm{grproj}(B))) \too E(\cD^b_\dg(\mathrm{grproj}(B)))\,.$$
\end{lemma}
\begin{proof}
Let $P \in \mathrm{grproj}(B)$. Note first that the exact endofunctor $\varphi\circ \psi$ of $\mathrm{grproj}(B)$ is given by $P \mapsto \bigoplus_{n \in \bbZ} \psi(P)_n \otimes_{B_0} B(-n)$. Since the functor $E$ is co-continuous, this yields the following equality
\begin{equation}\label{eq:equality1}
E(\varphi\circ \psi) = \sum_{n\in \bbZ} E(\psi(-)_n \otimes_{B_0}B(-n))\,.
\end{equation}
Given a finitely generated projective $\bbZ$-graded (right) $B$-module $P$ and an integer $m \in \bbZ$, let us write $F_m(P)$ for the $\bbZ$-graded (right) $B$-submodule of $P$ generated by the elements $\bigcup_{n \leq m} P_n$. In the same vein, given an integer $q\geq 0$, let us denote by $\mathrm{grproj}_q(B)$ the full subcategory of $\mathrm{grproj}(B)$ consisting of those $\bbZ$-graded (right) $B$-module $P$ such that $F_{-(q+1)}(P)=0$ and $F_q(P)=P$. Note that by definition we have an exhaustive increasing filtration
\begin{equation}\label{eq:filtration}
\mathrm{grproj}_0(B) \subset \mathrm{grproj}_1(B) \subset \cdots \subset \mathrm{grproj}_q(B) \subset \cdots \subset \mathrm{grproj}(B)\,.
\end{equation}
As explained by Quillen in \cite[pages 99-100]{Quillen}, for every $m \in \bbZ$, the assignment $P \mapsto F_m(P)/F_{m-1}(P)$ is an exact endofunctor of $\mathrm{grproj}(B)$. Moreover, we have a canonical isomorphism of exact functors between $\psi(-)_m \otimes_{B_0} B(-m)$ and $F_m(-)/F_{m-1}(-)$. Consequently, we obtain the following equality
\begin{equation}\label{eq:equality2}
\sum_{n\in \bbZ} E(\psi(-)_n \otimes_{B_0} B(-n)) = \sum_{n\in \bbZ} E(F_n(-)/F_{n-1}(-))\,.
\end{equation}
Now, note that every $\bbZ$-graded (right) $B$-module $P \in \mathrm{grproj}_q(B)$ admits a canonical filtration $0=F_{-(q+1)}(P) \subset \cdots \subset F_q(P)=P$. This yields a sequence $0=F_{-(q+1)}(-) \to \cdots \to F_q(-)=\id$ of exact endofunctors of $\mathrm{grproj}_q(B)$. Consequently, an inductive argument using the above general Proposition \ref{prop:additive} implies that the sum $\sum_{n=-q}^q E(F_n(-)/F_{n-1}(-))$ is equal to the identity of $E(\cD^b_\dg(\mathrm{grproj}_q(B)))$. Finally, using the fact that the above filtration \eqref{eq:filtration} of $\mathrm{grproj}(B)$ is exhaustive and that the functor $E$ is co-continuous, we hence conclude that
\begin{equation}\label{eq:equality3}
\sum_{n\in \bbZ} E(F_n(-)/F_{n-1}(-))=\id\,.
\end{equation}
The proof follows now from the combination of \eqref{eq:equality1} with \eqref{eq:equality2}-\eqref{eq:equality3}.
\end{proof}
Recall that $A$ is a (connected and locally finite-dimensional) $\bbN$-graded Noetherian $k$-algebra, which we assume to be Koszul and of finite global~dimension~$d$.
\begin{proposition}
We have a short exact sequence of dg categories
\begin{equation}\label{eq:sequence}
0 \too \cD^b_\dg(\mathrm{tors}(A)) \too \cD^b_\dg(\mathrm{gr}(A)) \too \cD^b_\dg(\mathrm{qgr}(A)) \too 0\,.
\end{equation}
\end{proposition}
\begin{proof}
As explained by Keller in \cite[Thm.~4.11]{ICM-Keller}, \eqref{eq:sequence} is a short exact sequence of dg categories if and only if the associated sequence of triangulated categories
\begin{equation}\label{eq:seq-triangulated}
\cD^b(\mathrm{tors}(A)) \too \cD^b(\mathrm{gr}(A)) \too \cD^b(\mathrm{qgr}(A))
\end{equation}
is exact sequence in the sense of Verdier. By definition, we have a short exact sequence of abelian categories $0 \to \mathrm{tors}(A) \to \mathrm{gr}(A) \to \mathrm{qgr}(A) \to 0$. Therefore, thanks to \cite[Lem.~1.15]{Exact} (consult also \cite{Grothendieck}), in order to show that \eqref{eq:seq-triangulated} is exact in the sense of Verdier, it suffices to prove the following condition: given a short exact sequence $0 \to L \to M \to N \to 0$ in the abelian category $\mathrm{gr}(A)$, with $L \in \mathrm{tors}(A)$, there exists a morphism of short exact sequences
$$
\xymatrix{
0 \ar[r] & L \ar@{=}[d] \ar[r] & M \ar[r] \ar[d] & N \ar[r] \ar[d] & 0 \\
0 \ar[r] & L \ar[r] & L' \ar[r] & L''\ar[r] & 0
}
$$
with $L'$ and $L''$ belonging to $\mathrm{tors}(A)$. Recall that the category $\mathrm{tors}(A)$ of torsion $A$-modules is defined as the full subcategory of $\mathrm{gr}(A)$ consisting of those $\bbZ$-graded (right) $A$-modules which are (globally) finite-dimensional over $k$. Given a $\bbZ$-graded (right) $A$-module $M$ and an integer $m \in \bbZ$, let us write $M_{\geq m}$ for the (right) $A$-submodule $\bigoplus_{n \geq m} M_n$ of $M$. Since by assumption $L$ is torsion and $M$ is finitely generated, there exists an integer $m \gg 0$ such that $L \cap M_{\geq m}=0$. Consequently, we can construct the following morphism of short exact sequences
$$
\xymatrix{
0 \ar[r] & L \ar@{=}[d] \ar[r] & M \ar[r] \ar[d] & N \ar[r] \ar[d] & 0 \\
0 \ar[r] & L \ar[r] & M/M_{\geq m} \ar[r] & M/\langle M_{\geq m}+L\rangle\ar[r] & 0\,.
}
$$
The proof follows now from the fact that, by construction, the $\bbZ$-graded (right) $A$-modules $M/M_{\geq m}$ and $M/\langle M_{\geq m}+L\rangle$ belong to $\mathrm{tors}(A)$.
\end{proof}
\begin{remark}\label{rk:triangle}
By assumption, the functor $E$ is localizing. Therefore, the short exact exact sequence of dg categories \eqref{eq:sequence} gives rise to a distinguished triangle:
$$
E(\cD^b_\dg(\mathrm{tors}(A))) \too E(\cD^b_\dg(\mathrm{gr}(A))) \too E(\cD^b_\dg(\mathrm{qgr}(A)))\stackrel{\partial}{\too} \Sigma E(\cD^b_\dg(\mathrm{tors}(A)))\,. $$
\end{remark}
Since $A$ has finite global dimension, the inclusion of categories $\mathrm{grproj}(A) \subset \mathrm{gr}(A)$ induces a Morita equivalence $\cD^b_\dg(\mathrm{grproj}(A)) \to \cD^b_\dg(\mathrm{gr}(A))$. Therefore, by first using the general Proposition \ref{prop:pgr} (with $B=A$) and then by applying the functor $E$ to the preceding Morita equivalence, we obtain an induced isomorphism
\begin{equation}\label{eq:induced2}
\oplus_{-\infty}^{+ \infty} E(k) \simeq E(\cD^b_\dg(\mathrm{grproj}(A))) \stackrel{\simeq}{\too} E(\cD^b_\dg(\mathrm{gr}(A)))\,.
\end{equation}
\begin{proposition}
We have a Morita equivalence
\begin{equation}\label{eq:BGG}
\cD^b_\dg(\mathrm{tors}(A)) \too \cD^b_\dg(\mathrm{grproj}(A^!))\,,
\end{equation}
where $A^!$ stands for the Koszul dual $k$-algebra of $A$.
\end{proposition}
\begin{proof}
Given a $\bbN$-graded $k$-algebra $B=\bigoplus_{n \geq 0} B_n$, let us denote by $\mathrm{Gr}(B)$ the category of all $\bbZ$-graded (right) $B$-modules and by $\cD(\mathrm{Gr}(B))$ the associated (unbounded) derived category. Following Beilinson-Ginzburg-Soergel \cite[\S2.12]{BGS}, let $\cD^\downarrow(\mathrm{Gr}(B))$, resp. $\cD^\uparrow(\mathrm{Gr}(B))$, be the full subcategory of $\cD(\mathrm{Gr}(B))$ consisting of those cochain complexes of $\bbZ$-graded (right) $B$-modules $M$ such that for some integer $m \gg 0$ we have $M^q_n\neq 0 \Rightarrow (q\geq -m\,\,\mathrm{or}\,\, q+n\leq m)$, resp. $M^q_n\neq 0 \Rightarrow (q\leq -m\,\,\mathrm{or}\,\, q+n\geq - m)$. These categories admit canonical dg enhancements $\cD_\dg(\mathrm{Gr}(B))$, $\cD_\dg^{\downarrow}(\mathrm{Gr}(B))$, and $\cD_\dg^{\uparrow}(\mathrm{Gr}(B))$. Now, recall from \cite[Thm.~2.12.1]{BGS} (consult also \cite[\S2]{Floystad}) the construction of the Koszul duality dg functor $\cD_\dg(\mathrm{Gr}(A)) \to \cD_\dg(\mathrm{Gr}(A^!))$. As proved in {\em loc. cit.}, this dg functor restricts to a Morita equivalence 
\begin{equation}\label{eq:Morita}
\cD_\dg^{\downarrow}(\mathrm{Gr}(A)) \too \cD_\dg^{\uparrow}(\mathrm{Gr}(A^!))
\end{equation}
which sends the $\bbZ$-graded (right) $A$-modules $k(i), i \in \bbZ$, to the $\bbZ$-graded (right) $A^!$-modules $\Sigma^{-i} A^!(i), i \in \bbZ$. Therefore, making use of the general Lemma \ref{lem:Orlov} below (with $B=A$ and $B=A^!$), we conclude that \eqref{eq:Morita} restricts furthermore to the above Morita equivalence \eqref{eq:BGG}.
\end{proof}
\begin{lemma}\label{lem:Orlov}
Let $B=\bigoplus_{n\geq 0}B_n$ be a (connected and locally finite-dimensional) $\bbN$-graded Noetherian $k$-algebra. The smallest thick triangulated subcategory of $\cD^b(\mathrm{gr}(B))$ containing the $\bbZ$-graded (right) $B$-modules $\{k(i)\,|\,i \in \bbZ\}$, resp. $\{B(i)\,|\,i \in \bbZ\}$, agrees with $\cD^b(\mathrm{tors}(B))$, resp. $\cD^b(\mathrm{grproj}(B))$.
\end{lemma}
\begin{proof}
Consult the proof of \cite[Lem.~2.3]{Orlov}.
\end{proof}
Recall that since $A$ is connected, its Koszul dual $k$-algebra $A^!$ is also connected. Therefore, by first applying the functor $E$ to \eqref{eq:BGG} and then by using the above general Proposition \ref{prop:pgr} (with $B=A^!$), we obtain an induced isomorphism
\begin{equation}\label{eq:induced1}
E(\cD^b_\dg(\mathrm{tors}(A))) \stackrel{\simeq}{\too} E(\cD^b_\dg(\mathrm{grproj}(A^!)))\simeq \oplus_{-\infty}^{+ \infty} E(k)\,.
\end{equation}
Since $A$ is Koszul and of finite global dimension $d$, we have a linear free resolution
\begin{equation}\label{eq:resolution}
0 \too A(-d)^{\oplus \beta_d} \too \cdots \too A(-2)^{\oplus \beta_2} \too A(-1)^{\oplus \beta_1} \too A \too k \too 0
\end{equation}
of the $\bbZ$-graded (right) $A$-module $k$. As mentioned in \S\ref{sec:intro}, the integer $\beta_i$ agrees with the dimension of the $k$-vector space $\mathrm{Tor}^A_i(k,k)$ (or $\mathrm{Ext}_A^i(k,k)$).
\begin{proposition}\label{prop:triangle}
Under the above isomorphisms \eqref{eq:induced2} and \eqref{eq:induced1}, the distinguished triangle of Remark \ref{rk:triangle} identifies with
\begin{equation}\label{eq:triangle2}
\oplus^{+\infty}_{-\infty} E(k) \stackrel{\mathrm{M}'}{\too} \oplus^{+\infty}_{-\infty} E(k) \too E(\cD^b_\dg(\mathrm{qgr}(A))) \stackrel{\partial}{\too} \oplus^{+\infty}_{-\infty} \Sigma E(k) \,,
\end{equation}
where $\mathrm{M}'$ stands for the (infinite) matrix $\mathrm{M}'_{ij}:= (-1)^j (-1)^{(i-j)} \beta'_{i-j}$.
%
\end{proposition}
\begin{proof}
Let $\NMot(k)$ be the category of {\em noncommutative motives} constructed in \cite[\S 8.2]{book}; denoted by $\NMot(k)_{\mathrm{loc}}$ in {\em loc. cit.} By construction, this triangulated category comes equipped with a functor $U\colon \dgcat(k) \to \NMot(k)$ which is initial among all the functors satisfying conditions (i)-(iii). Concretely, given a functor $E\colon \dgcat(k) \to \cT$ satisfying conditions (i)-(iii), there exists a (unique) triangulated functor $\overline{E}\colon \NMot(k) \to \cT$ such that $\overline{E}\circ U \simeq E$. Moreover, $\overline{E}$ preserves arbitrary direct sums; see \cite[Thm.~8.5]{book}. This implies that in order to prove Theorem \ref{prop:triangle}, it suffices to show that the triangle of Remark \ref{rk:triangle} (with $E=U$) identifies with
\begin{equation}\label{eq:triangle22}
\oplus^{+\infty}_{-\infty} U(k) \stackrel{\mathrm{M}}{\too} \oplus^{+\infty}_{-\infty} U(k) \too U(\cD^b_\dg(\mathrm{qgr}(A))) \stackrel{\partial}{\too} \oplus^{+\infty}_{-\infty} \Sigma U(k) \,,
\end{equation}
where $\mathrm{M}$ stands for the (infinite) matrix $\mathrm{M}_{ij}:= (-1)^j (-1)^{(i-j)} \beta_{i-j}$. Recall from \cite[\S8.6]{book} that, for every dg category $\cA$, we have a natural isomorphism 
$$\Hom_{\NMot(k)}(U(k),U(\cA))\simeq K_0(\cA)\,.$$ 
Moreover, $U(k)$ is a compact object of the triangulated category $\NMot(k)$. Therefore, since $K_0(k)\simeq \bbZ$, an endomorphism of $\oplus_{-\infty}^{+ \infty} U(k)$ corresponds to an infinite matrix with integer coefficients in which every column has solely a finite number of non-zero entries. Let us denote by $\mathrm{M}$ the matrix corresponding to $U(\cD^b_\dg(\mathrm{tors}(A))) \to U(\cD^b_\dg(\mathrm{gr}(A)))$ under the isomorphisms \eqref{eq:induced2} and \eqref{eq:induced1} (with $E=U$). By applying the functor $\Hom_{\NMot(k)}(U(k),-)$ to the isomorphisms \eqref{eq:induced2} and \eqref{eq:induced1} (with $E=U$), we obtain induced abelian group isomorphisms
\begin{equation}\label{eq:iso2}
\oplus_{- \infty}^{+\infty} \bbZ \simeq K_0(\cD^b(\mathrm{grproj}(A))) \stackrel{\simeq}{\too} K_0(\cD^b(\mathrm{gr}(A)))
\end{equation}
\begin{equation}\label{eq:iso1}
K_0(\cD^b(\mathrm{tors}(A))) \stackrel{\simeq}{\too} K_0(\cD^b(\mathrm{grproj}(A^!))) \simeq \oplus_{- \infty}^{+\infty} \bbZ\,.
\end{equation}
The element $1 \in \bbZ$, placed at the $j^{\mathrm{th}}$ component of the direct sum $\oplus_{-\infty}^{+\infty} \bbZ$, corresponds under \eqref{eq:iso1} to the Grothendieck class $[\Sigma^{-j}k(-j)] = (-1)^j [k(-j)] \in K_0(\cD^b(\mathrm{tors}(A)))$. In the same vein, the element $1\in \bbZ$, placed at the $i^{\mathrm{th}}$ component of the direct sum $\oplus_{-\infty}^{+ \infty}\bbZ$, corresponds under \eqref{eq:iso2} to the Grothendieck class $[A(-i)] \in K_0(\cD^b(\mathrm{gr}(A)))$. Thanks to the above linear free resolution \eqref{eq:resolution}, we have moreover the following equality $[k(-j)] = \sum_{i=0}^d (-1)\beta_i [A(-i-j)]$ in the Grothendieck group $K_0(\cD^b(\mathrm{gr}(A)))$. The above considerations allow us to conclude that the $(i,j)^{\mathrm{th}}$ entry of the matrix $\mathrm{M}$ is given by the integer $(-1)^j (-1)^{(i-j)} \beta_{i-j}$. This finishes the proof.
\end{proof}
We now have all the ingredients necessary for the conclusion of the proof of Theorem \ref{thm:main}(i). Let $o\in \cT$ be a compact object. By applying the functor $\Hom_\cT(o,-)$ to the triangle \eqref{eq:triangle2}, we obtain an induced long exact sequence of $R$-modules:
\begin{equation*}\label{eq:long}
\cdots \to \oplus^{+\infty}_{-\infty} E^o_m(k) \stackrel{\mathrm{M}'}{\too} \oplus^{+\infty}_{-\infty} E^o_m(k) \to E^o_m(\cD^b_\dg(\mathrm{qgr}(A))) \stackrel{\partial}{\to} \oplus^{+\infty}_{-\infty} E^o_{m-1}(k)\to \cdots\,.
\end{equation*}
Since $\mathrm{M}'_{ij}=(-1)^j(-1)^{(i-j)} \beta'_{i-j}$, with $\beta'_0=1$ and $\beta'_r=0$ whenever $r \notin \{0, \ldots, d'\}$, a simple matrix computation shows that the preceding homomorphism $\mathrm{M}'$ of $R$-modules is injective. Consequently, the long exact sequence breaks-up into short exact sequences of $R$-modules:
\begin{equation}\label{eq:short}
0 \too \oplus^{+\infty}_{-\infty} E^o_m(k) \stackrel{\mathrm{M}'}{\too} \oplus^{+\infty}_{-\infty} E^o_m(k) \too E^o_m(\mathrm{qgr}(A)) \too 0\,.
\end{equation}
Thanks to Lemma \ref{lem:final} below and to the definition of the homomorphism $\phi$ (see below), we also have the following short exact sequences of $R$-modules:
\begin{equation*}\label{eq:short1}
0 \to R[t,t^{-1}] \otimes_R E^o_m(k) \stackrel{\phi \otimes \id}{\to} R[t,t^{-1}]\otimes_R E^o_m(k) \to R[t]/\langle h'_A(t)^{-1}\rangle \otimes_R E^o_m(k) \to 0\,.
\end{equation*}
Now, consider the Poincar\'e polynomial $p_A(t):=\sum^d_{i=0} (-1)^i \beta_i t^i$ (and $p'_A(t):=\sum_{i=0}^{d'} (-1)^i \beta'_i t_i$). Thanks to the linear free resolution \eqref{eq:resolution}, we have $h_A(t)^{-1}=p_A(t)$ (and $h'_A(t)^{-1}=p'_A(t)$). This implies that under the canonical isomorphism between $\oplus_{-\infty}^{+\infty} E^o_m(k)$ and $R[t,t^{-1}]\otimes_R E^o_m(k)$, the matrix $\mathrm{M}'$ corresponds to the homomorphism $\phi\otimes \id$. Consequently, we obtain induced $R$-module isomorphisms 
\begin{eqnarray}\label{eq:induced-22}
E^o_m(\cD^b_\dg(\mathrm{qgr}(A))) \stackrel{\simeq}{\too} R[t]/\langle h'_A(t)^{-1}\rangle \otimes_R E^o_m(k) && m \in \bbZ\,.
\end{eqnarray}
This concludes the proof of Theorem \ref{thm:main}(i).
\begin{lemma}\label{lem:final}
We have the following short exact sequence of $R$-modules
$$ 0 \too R[t, t^{-1}] \stackrel{\phi}{\too} R[t,t^{-1}] \too R[t]/\langle h'_A(t)^{-1}\rangle \too 0\,,$$
where $\phi$ stands for the homomorphism $p(t) \mapsto p(-t)\cdot h'_A(t)^{-1}$.
\end{lemma}
\begin{proof}
Since $h'_A(0)^{-1}=1$, the homomorphism $\phi$ is injective. Moreover, we have the following natural isomorphisms
$$ \mathrm{cokernel}(\phi)= R[t,t^{-1}]/\mathrm{Im}(\phi) \stackrel{(a)}{\simeq} R[t,t^{-1}]/\langle h'_A(t)^{-1}\rangle \stackrel{(b)}{\simeq} R[t]/\langle h'_A(t)^{-1}\rangle\,,$$
where $(a)$ follows from the fact that the homomorphisms $\phi$ and $-\cdot h'_A(t)^{-1}$ have the same image, and $(b)$ from the fact that the polynomial $t$ is invertible in $R[t]/\langle h'_A(t)^{-1}\rangle$ (this follows from the fact that $h'_A(0)^{-1}=1$). This concludes the proof.
\end{proof}
We now have all the ingredients necessary for the conclusion of the proof of Theorem \ref{thm:main}(ii). Consider the following composition 
\begin{equation}\label{eq:induced}
\oplus_{n=0}^{d'-1} E(k)\too \oplus_{-\infty}^{+\infty} E(k) \too E(\cD^b_\dg(\mathrm{qgr}(A)))\,.
\end{equation}
By assumption, the triangulated category $\cT$ is compactly generated. Therefore, the morphism \eqref{eq:induced} is invertible if and only if for every compact object $o \in \cT$ the induced $R$-module homomorphisms
\begin{eqnarray}\label{eq:induced11}
\oplus_{n=0}^{d'-1} E^o_m(k) \too E^o_m(\cD^b_\dg(\mathrm{qgr}(A))) && m \in \bbZ
\end{eqnarray}
are invertible. Under the canonical identification $\oplus_{n=0}^{d'-1}R \otimes_R E^o_m(k)\simeq \oplus_{n=0}^{d'-1}E^o_m(k)$, the composition of \eqref{eq:induced11} with \eqref{eq:induced-22} corresponds to the $R$-module homomorphisms:
\begin{eqnarray*}
\big((1, t, \ldots, t^{d'-1}) \colon \oplus_{n=0}^{d'-1} R \too R[t]/\langle h'_A(t)^{-1}\rangle\big)\otimes_R E^o_m(k) && m \in \bbZ\,.
\end{eqnarray*}
By assumption, we have $1/\beta' \in R$. Therefore, the factorization algorithm for polynomials applied to $R[t]$ allows us to conclude that the $R$-module homomorphism $(1, t, \ldots, t^{d'-1})$ is invertible. This implies that the induced $R$-module homomorphisms \eqref{eq:induced11} are also invertible, and so the proof of Theorem \ref{thm:main}(ii) is finished.

\medbreak\noindent\textbf{Acknowledgments:} The author is grateful to Michael Artin for useful discussions concerning noncommutative projective schemes and also to Theo Raedschelders for important comments on a previous version of this note.

\end{document}

\end{proof}